\documentclass[11pt]{article}

\usepackage{amsmath}
\usepackage{amssymb}

 \newcommand{\beq}{\begin{equation}}
\newcommand{\eeq}{\end{equation}}

\newtheorem{theorem}{Theorem}[section]
\newtheorem{proposition}[theorem]{Proposition}
\newtheorem{lemma}[theorem]{Lemma}
\newtheorem{corollary}[theorem]{Corollary}

\newtheorem{example}[theorem]{Example}

\newenvironment{proof}{\medbreak\noindent{\it Proof.}\rm}{\hfill$\square$\rm}

\newcommand{\vph}{\varphi}

\newcommand{\R}{{ \mathbb R}}
\newcommand{\Rn}{{ \mathbb R}^n}

\newcommand{\C}{{\mathbb  C}}
\newcommand{\Cn}{{\mathbb  C}^n}

\newcommand{\D}{{\mathbb D}}

\newcommand{\BE}{{\mathbf E}}

\newcommand{\F}{{\mathcal F}}
\newcommand{\E}{{\mathcal E}}
\newcommand{\M}{{\mathcal M}}

\newcommand{\PSH}{{\operatorname{PSH}}}

\newcommand{\Exp}{{\operatorname{Exp}\,}}

\newcommand{\Capa}{{\operatorname{Cap}\,}}

\begin{document}


\begin{center}
{\Large\bf Local geodesics for plurisubharmonic functions}
\end{center}

\begin{center}
{\large Alexander Rashkovskii}
\end{center}

\vskip1cm

\begin{abstract}
We study geodesics for plurisubharmonic functions from the Cegrell class ${\mathcal F}_1$ on a bounded hyperconvex domain of ${\mathbb  C}^n$ and show that, as in the case of metrics on K\"{a}hler compact menifolds, they linearize an energy functional. As a consequence, we get a uniqueness theorem for functions from ${\mathcal F}_1$ in terms of total masses of certain mixed Monge-Amp\`ere currents. Geodesics of relative extremal functions are considered and a reverse Brunn-Minkowski inequality  is proved for capacities of multiplicative combinations of multi-circled compact sets. We also show that functions with strong singularities generally cannot be connected by (sub)geodesic arks.
\end{abstract}

\section{Introduction}
Starting with pioneer work by Mabuchi \cite{Mab}, a notion of geodesics in the space of K\"{a}hler metrics on compact complex manifolds has been playing a prominent role in K\"{a}hler geometry and has found a lot of applications. We will not give here any detailed account on this subject;  the interested reader can consult, for example, \cite{Sem}, \cite{Chen}, \cite{Do}, \cite{BmBo1}, \cite{Be1}, \cite{G12}, and the bibliography therein.
In particular, geodesics in the space of metrics on a compact $n$-dimensional K\"{a}hler manifold $(X,\omega)$ have been characterized as solutions to a complex homogeneous equation, which implies linearity of the Mabuchi functional
\beq\label{Mab0} \M(\psi,\phi_0)=\frac1{n+1}\int_X (\psi-\phi_0)\sum_{k=0}^n (dd^c \psi)^k\wedge(dd^c \phi_0)^{n-k}\eeq
along the geodesics $\psi=\psi_t$ (here $\phi_0$ is a reference metric).

We believe however that a local, flat situation of functions on a bounded pseudoconvex domain $D$ of $\Cn$ deserves independent consideration, at least because of possible applications. The simplest choice here are functions with zero boundary values on $\partial D$ and finite total Monge-Amp\`ere mass. To provide existence of the corresponding boundary problem on $D\times \{1<|\zeta|<e\}$, we require also finiteness of the Monge-Amp\`ere energy $\BE(u)=\int_D u(dd^cu)^n$. For such (not necessarily bounded) plurisubharmonic functions we show in Theorem~\ref{mabE1} that the energy functional $u\mapsto\BE(u)$
plays role of the Mabuchi functional (\ref{Mab0}). We use this in proving a uniqueness result (Theorem~\ref{uniq} and Corollary~\ref{uniq1}) for  functions from the Cegrell class $\F_1(D)$ in terms of total masses of $n+1$ mixed Monge-Amp\`ere currents on $D$.

We discuss briefly geodesics connecting relative extremal functions $\omega_{K_j}$ of compact subsets $K_j$ of $D$. In the multi-circled case,
a variant of reversed Brunn-Minkowski inequality is proved for the Monge-Amp\`ere capacities of multiplicative combinations of $K_j$. We
present a simple example where the geodesic functions $u_t$ are still relative extremal functions, however not of compact sets but of multi-plate condensers.

The case of bounded functions (Theorem~\ref{geodrel}) is close to the classical setting of K\"{a}hler metrics, with a modification to handle the boundary effects. The general case requires a justification for existence of solutions of the corresponding boundary problem like that in \cite{BBGZ13} and \cite{Da14a}. We show that while this works for $\F_1(D)$ (Theorem~\ref{mabE1}), for functions with strong singularities (say, with positive Lelong numbers) such a problem generally has no solution (Theorem~\ref{minenv}).

\medskip

\section{Energy functional on Cegrell classes}

Let $D\subset \Cn$ be a bounded hyperconvex domain. We recall that Cegrell's class $\E_0(D)$ consists of bounded plurisubharmonic functions $u$ in $D$ with zero boundary values on $\partial D$ and finite total Monge-Amp\`ere mass
$$\int_D(dd^cu)^n<\infty;$$
class $\E_1(D)$ consists of functions $u$ that are limits of decreasing sequences $u_j\in\E_0(D)$ such that
$$\sup_j \int_D |u_j| (dd^cu_j)^n<\infty;$$
if, in addition,
$$\sup_j \int_D (dd^cu_j)^n<\infty,$$
then $u\in\F_1(D)$.

If $u\in\E_1(D)$, then the current $(dd^cu)^n$ is defined as the limit of $(dd^cu_j)^n$ and is independent of the choice of the approximating sequence $u_j$ \cite[Thm. 3.8]{Ce03}.

For any function $u\in\E_1(D)$, consider its energy functional
\beq\label{Mab1} \BE(u)=(n+1)\,\M(u,0)=\int_D u(dd^c u)^n.\eeq
For any sequence $u_j$ from the definition of $\E_1(D)$, we have
$\BE(u_j)\to \BE(u)$ \cite[Thm. 3.8]{Ce03}.

Similarity with the Mabuchi functional (\ref{Mab0}) for metrics on compact manifolds becomes visible from the following important identity.

\begin{proposition}\label{lem:Mab} For any $u,v\in\E_1(D)$,
\beq\label{Mab} \BE(u)-\BE(v)=\int_D (u-v)\sum_{k=0}^n (dd^c u)^k\wedge(dd^c v)^{n-k}.\eeq
\end{proposition}

\begin{proof} This easily follows from the integration by parts formula
\beq\label{ibp}\int_D u\, dd^c v\wedge T= \int_D v\, dd^c u\wedge T\eeq
valid for $u,v\in\E_1$ and positive closed currents $T$ \cite[Cor. 3.4]{Ce04}.
\end{proof}

\begin{corollary}\label{dompr} If $u,v\in\E_1(D)$ satisfy $u\le v$, then $\BE(u)\le \BE(v)$. If, in addition, $u\in\F_1(D)$ and $\BE(u)= \BE(v)$, then $u= v$ on $D$.
\end{corollary}

\begin{proof} The inequality is well known (see, for example, \cite[Thm. 3.8]{Ce03}) and follows, in particular, directly from Proposition~\ref{lem:Mab}.

 The condition $\BE(u)=\BE(v)$ gives us, by (\ref{Mab}), $(dd^c u)^n=0$ on the set $A=\{z:\: u(z)<v(z)\}$. We claim that this implies $u=v$ everywhere in $D$. In \cite{Ce}, this was proved for locally bounded $u$ and $v$; we adapt the proof to our case.
Let $P(z)=|z|^2-C\in\PSH^-(D)$.
 If $u(z_0)<v(z_0)$, then the set $A_\eta=\{z:\: u(z)<\eta P(z)+v(z)\}$ has positive Lebesgue measure for some $\eta>0$.

 By \cite[Lemma~4.4]{Ce03},
 $$ \eta^n\int_{A_\eta}(dd^cP)^n \le \int_{A_\eta}(dd^c(\eta P+v)^n)\le \int_{A_\eta}(dd^c u)^n \le \int_{\{u<v\}} (dd^cu)^n=0,$$
which contradicts the positivity of the Lebesgue measure of $A_\eta$.
\end{proof}

\medskip
\noindent{\bf Remark.} The second statement of Corollary~\ref{dompr} remains true if the condition $u\in\F_1(D)$ is replaced by $u\in\E_1(D)$ and $\lim u(z)=0$ as $z\to\partial D$. In this case (increasing, if needed, the constant $C$ in the definition of the function $P$), the set $A_\eta$ is compactly supported in $D$ and thus both $u$ and $v$ have finite Monge-Amp\`ere mass on a neighborhood of $\overline A_\eta$, so \cite[Lemma~4.4]{Ce03} still can be applied.

\section{Geodesics for the class $\E_0$}

Let $S$ be the annulus $\{\zeta\in\C:\: 1<|\zeta|<e\}$ bounded by the circles $S_0=\{|\zeta|=1\}$ and $S_1=\{|\zeta|=e\}$.
Given two functions $u_0,u_1\in\E_0(D)$, consider the class $W(u_1,u_2)$ of all functions $u\in\PSH^-(D\times S)$ such that $\limsup\, u(z,\zeta)\le u_j(z)$ as ${\zeta\to S_j}$.  The class is not empty because, for example, is contains $u_0+u_1$.

Denote $$\widehat u(z,\zeta)=\sup\{u(z,\zeta):\: u\in W(u_1,u_2)\}.$$
Since its u.s.c.  regularization ${\widehat u}^*$ belongs to $W(u_1,u_2)$, we have $\widehat u={\widehat u}^*$. Moreover, being a maximal plurisubharmonic function, it satisfies the homogeneous Monge-Amp\`ere equation
\beq\label{relextMA}(dd^c \widehat u)^{n+1}=0\ {\rm on\ }D\times S.\eeq

Evidently, $\widehat u(z,\zeta)=\widehat u(z,|\zeta|)$ on $D\times S$, so the function $u_t(z):=\widehat u(z,e^t)$ is convex in $t\in(0,1)$; we will call it the {\it geodesic} of $u_0$ and $u_1$. Similar to \cite{Be1}, we get

 \begin{proposition}\label{relextbounds} The geodesic $u_t$ of $u_0,u_1\in\E_0(D)$ has the following properties:
\begin{enumerate}
\item[(i)]  $u_t(z)\to 0$ as $z\to\partial D$;
\item[(ii)] $u_t\to u_j$ as $t\to j$, uniformly on $D$ ($j=0,1$);
\item[(iii)]  $ u_t\le U_t:=(1-t)u_0+tu_1$;
\item[(iv)] $u_t\ge s_t:=\max\{u_0-M_1\,t,u_1-M_0\,(1-t)\}$, where $M_j=\|u_j\|_\infty$.
\end{enumerate}
\end{proposition}

\begin{proof} Since $u_t\ge u_0+u_1$, we have (i).
Relation (iii) follows because $U_0=u_0$, $U_1=u_1$ and $U_t$ is harmonic in $t$ (while $u_t$ is convex in $t$). The lower bound (iv) is evident because $\widehat s(z,\zeta):=s_{\log|\zeta|}(z)$ belongs to $W(u_0,u_1)$. Finally, (iii) and (iv) imply (ii).
\end{proof}

\medskip

A family of functions $v_t\in\E_0(D)$, $0<t<1$, will be called a {\it subgeodesic} for $u_0$ and $u_1$ if $\widehat v(z,\zeta):=v_{\log|\zeta|}(z)\in W(u_0,u_1)$.

\medskip

Let us study values of the energy functional $\BE$ on curves in $\E_0(D)$. Here again we get its properties as in the case of compact manifolds.

\medskip
\begin{proposition}\label{conc} The functional $v\mapsto \BE(v)$ is concave on $\E_0(D)$.
\end{proposition}

\begin{proof} Let
$U_t=(1-t)u_0+tu_1$, $0<t<1$. By Proposition~\ref{lem:Mab},
$$\frac{d}{dt}\, \BE(U_t)= (n+1)\int_D (u_1-u_0) (dd^c U_t)^n,$$
so
\begin{eqnarray*}
\frac1{n+1}\,\frac{d^2}{dt^2} \BE(U_t) &=& n\int_D (u_1-u_0)\wedge dd^c(u_1-u_0)\wedge (dd^c U_t)^{n-1}\\
 &=& - n\int_D d(u_1-u_0)\wedge d^c(u_1-u_0)\wedge (dd^c U_t)^{n-1}\le 0,
\end{eqnarray*}
which proves the claim.
\end{proof}

\medskip

It also turns out that, on the other hand, the function $\BE(v_t)$ is convex along subgeodesics.

\begin{theorem}\label{geodrel}
Let $v_t$ be a subgeodesic for $u_0,u_1\in \E_0(D)$. Then the function $t\mapsto \BE(v_t)$ is convex, and it is linear if and only if the subgeodesic $v_t$ is a geodesic.
\end{theorem}

\begin{proof} The idea of the proof is similar to that for Proposition~\ref{conc}, however it needs more technicalities.

Convexity of $\BE(v_t)$ is equivalent to subharmonicity of the function
$$\widehat \BE= \BE(\widehat v)=\int_D \widehat v(d_z^{}d_z^c \widehat v)^n,$$ and the linearity of $\BE$ corresponds to the harmonicity of $\widehat \BE$.
The corresponding result for the Mabuchi functional (\ref{Mab0}) on a compact manifold $X$ follows from the formula
\beq\label{BmBo} d_\zeta^{}d_\zeta^c \widehat \E =\int_X (dd^c \widehat v)^{n+1}
\eeq
(see, for example, \cite{BmBo1}), and one gets then the claims from the plurisubharmonicity of the subgeodesics and equation (\ref{relextMA}).

In the case of functions from $\E_0(D)$, $D\subset\Cn$, one can argue as follows.
By \cite[Thm. 1.2]{Ce09}, $\widehat v$ is the limit of a decreasing sequence of smooth functions $\widehat v^{(j)}$ from $\E_0(D\times S)$; clearly, they can be assumed to be independent of the argument of $\zeta$. Furthermore, since $v_t^{(j)}\in \E_0(D)$ decrease to $v_t\in\E_0(D)$, we have $\BE(v_t^{(j)})\to \BE(v_t)$ by \cite[Thm. 3.8]{Ce03}. So, we can assume $\widehat v\in \E_0(D\times S)\cap C^{\infty}(D\times S)$.

Note that the aforementioned approximation theorem rests on the following result from \cite{Gu}, see also \cite[Lem. 2.2]{Ce09}:
{\sl If $\vph,\psi\in\PSH(\Omega)$ and $b:\,\R\to\R_+$ is a smooth convex function with $b(x)=|x|$ for all $|x|>\epsilon>0$, then $\max_b(\vph,\psi):=\vph+\psi+b(\vph-\psi)\in\PSH(\Omega)$.}

If we take here $\Omega=D\times S$, $\vph=\widehat v -2\epsilon$, and $\psi=\rho/\epsilon$ for a smooth exhaustion function $\rho$ of $D$ (which exists by \cite[Cor. 1.3]{Ce09}), then $\max_b(\vph,\psi)\in \E_0(D\times S)\cap C^{\infty}(D\times S)$. Moreover, it coincides with $\rho/\epsilon$ near $\partial D\times S$, so it is independent of $\zeta$ there. Since $\max_b(\vph,\psi)\to \widehat v$ uniformly as $\epsilon\to 0$, we can thus also assume $d_\zeta\widehat v=0$ near $\partial D$.

\medskip
By Proposition~\ref{lem:Mab},
$$ d_\zeta^c \widehat \BE=(n+1)\int_D d_\zeta^c \widehat v\wedge (d_z^{}d_z^c\widehat v)^n,$$
so
\begin{eqnarray*}
 \frac1{n+1}\, d_\zeta^{} d_\zeta^c \widehat \BE &=& \int_D d_\zeta^{} d_\zeta^c \widehat v\wedge (d_z^{}d_z^c\widehat v)^n
+ n\int_D d_\zeta^c \widehat v\wedge d_\zeta^{}(d_z^{}d_z^c\widehat v)\wedge(d_z^{}d_z^c\widehat v)^{n-1} \\
&=& \int_D d_\zeta^{} d_\zeta^c \widehat v\wedge (d_z^{}d_z^c\widehat v)^n
- n\int_D d_z^{}d_\zeta^c \widehat v\wedge d_z^cd_\zeta^{}\widehat v\wedge(d_z^{}d_z^c\widehat v)^{n-1}\\
&=&  \frac1{n+1}\,\int_D (dd^c\widehat v)^{n+1},
\end{eqnarray*}
where the second equality follows from Stokes' theorem because $d_\zeta\widehat v=0$ near $\partial D$, and the last one by direct calculation with $d=d_z^{}+d_\zeta^{}$, $d^c=d_z^c+d_\zeta^c$.

Finally, let $v_j=\lim v_t$ as $t\to j$ for $j=0,1$, and let $w_t$ be the geodesic of $v_0, v_1$. If $\BE(v_t)$ is linear, then $\BE(v_t)=\BE(w_t)$, so $v_t=w_t$ for all $t$ by Corollary~\ref{dompr}.
\end{proof}

\medskip
Now we can prove the following uniqueness result.

\begin{theorem}\label{uniq}
Let $u_0,u_1\in\E_0(D)$ satisfy
\beq\label{balance} \int_D u_0 (dd^c u_0)^k\wedge (dd^c u_1)^{n-k} =\BE(u_1),\quad k=0,\ldots,n.\eeq
Then $u_0= u_1$ in $D$.
\end{theorem}

\begin{proof} By (\ref{ibp}), condition (\ref{balance}) implies
$$ \int_D u_1 (dd^c u_0)^k\wedge (dd^c u_1)^{n-k} =\BE(u_1),\quad k=0,\ldots,n,$$
as well, so
\beq\label{Ut0} \int_D (u_1-u_0) (dd^c u_0)^k\wedge (dd^c u_1)^{n-k} =0,\quad k=0,\ldots,n.\eeq
Denote $U_t=(1-t)u_0+tu_1$. By (\ref{Ut0}) and a computation in the proof of Proposition~\ref{conc}, the function $\BE(U_t)$ is linear on
$[0,1]$, so $\BE(U_t)= \BE(u_0)$.

On the other hand, by Proposition~\ref{relextbounds}, the geodesic $u_t$ of $u_0$ and $u_1$ satisfies $u_t\le U_t$ and, by Theorem~\ref{geodrel}, $\BE(u_t)=\BE(u_0)$ as well. By Corollary~\ref{dompr}, we get $u_t=U_t$ for any $t$.

Therefore, the function $\widehat U(z,\zeta) =(1-\log|\zeta|)\, u_0(z)+\log|\zeta|\, u_1(z)$ is plurisubharmonic in $D\times S$. Then
$$\frac\partial{\partial\bar z_k}(u_1-u_0)=0$$
for all $k$, so $u_1-u_0$ is analytic in  $D$, equal to $0$ on $\partial D$, and thus is identical $0$.
\end{proof}

\bigskip

\noindent{\bf Remark.} If $u\in\E_0(D)$ and $u_j=\max\{u,-\alpha_j\}$, then we have
$$ \int_D  (dd^c u_0)^k\wedge (dd^c u_1)^{n-k} =\int_D  (dd^c u_1)^n,\quad k=0,\ldots,n,$$
for any $\alpha_0,\alpha_1>0$. Therefore, using the mixed energy functionals in Theorem~\ref{uniq} is essential.


\section{Example: geodesics of relative extremal functions}

Here we consider a particular case of the construction above.
Recall that the {\it relative extremal function} of a set $K\Subset D$ is
$$\omega_{K}^{}(z)=\limsup_{x\to z}\,\sup\{u(x):\: u\in\PSH^-(D),\ u|_K\le -1\}\in\E_0(D).$$

We will be interested in the following: {\sl Given two relatively compact subsets $K_0$ and $K_1$ of $D$, let $u_j=\omega_{K_j}$ for $j=1,2$, what can be said about their geodesic $u_t$? In particular, is $u_t$ for any fixed $t$ a relative extremal function on $D$ and if not, how far is it from being such?}

\medskip
Note that
\beq\label{Ecap} \BE(\omega_{K})=\int_D \omega_K (dd^c\omega_K)^n=-\int_D  (dd^c\omega_K)^n=-\Capa(K),\eeq
the Monge-Amp\`ere capacity of $K$ with respect to $D$.
We have, by Theorem~\ref{geodrel}, the following

\begin{proposition}\label{linomega} If $u_t$ is the geodesic for a pair of relative extremal functions $\omega_{K_j}$, then
$$\BE(u_t)=(t-1)\,\Capa(K_0) -t\,\Capa(K_1).$$
\end{proposition}

\medskip
Denote
$L_t=\{z\in D:\: u_t(z)=-1\}$,
then we have
$ u_t\le \omega_{L_t}$. By Corollary~\ref{dompr} and Proposition~\ref{linomega}, this implies

\begin{proposition}\label{linomega1}
In the conditions of Proposition~\ref{linomega},
$$\Capa(L_t)\le (1-t)\,\Capa(K_0)+t\,\Capa(K_1),$$
 and the inequality becomes equality if and only if $\omega_{L_t}$ is the geodesic.
 \end{proposition}


\medskip

Now let us assume $D$ to be a bounded complete logarithmically convex Reinhardt domain of $\Cn$, that is, $y\in D$ provided $z\in D$ and $|y_l|\le |z_l|$ for all $l$, and such that the set $\log D=\{s\in\Rn_-:\: \Exp s\in D_j\}$ is a convex subset of $\Rn$; here $\Exp s= (e^{s_1},\ldots, e^{s_n})$. In addition, let $K_j$, $j=0,1$, be compact Reinhardt subsets of $D$.
In this setting, the functions $\omega_{K_j}$ are toric (multi-circled) and so, the function
$$\check u(s,t):=u_t(\Exp s)$$
is convex in $(s,t)\in \Rn_-\times (0,1)$.
Denote
\beq\label{kt}K_t=K_0^{1-t}K_1^t=\{z\in\D^n: |z_l|=|\eta_l|^{1-t} |\xi_l|^{t}, \ 1\le l\le n,\ \eta\in K_0,\ \xi\in K_1\}, \quad 0<t<1;\eeq
in other words,  $\log K_t=(1-t)\log K_0+t\log K_1$. Note that $K_t\subset D$ because $\log D$ is convex.

Recall that volumes $|\cdot|$ of convex combinations $(1-t)P_0+t\,P_1$ of two bodies $P_j\subset \Rn$  satisfy
$$ |(1-t)P_0+t\,P_1|\ge |P_0|^{1-t}\,|P_1|^t,$$
the Brunn-Minkowski inequality (in multiplicative form).
In our case, the sets $\log K_j$ typically are of infinite volume. Instead of the volumes, we have
a {\sl reversed Brunn-Minkowski inequality for the capacities} of $K_t$ (multiplicative combinations of $K_j$), in additive form.

\begin{theorem} In the Reinhardt situation, the capacities of the sets $K_t$ defined by (\ref{kt}) satisfy
$$\Capa(K_t)\le (1-t)\,\Capa(K_0)+t\,\Capa(K_1).$$
\end{theorem}

\begin{proof}
By the convexity of $\check u$, we have $\check u(s,t)\le -1$ when $s\in (1-t)\log K_0+t\log K_1= \log K_t$.
Therefore, $K_t\subset L_t$, and the result follows from Proposition~\ref{linomega1}.
\end{proof}

\medskip

 Evidently, $\omega_{K_t}^{}$ is the geodesic if and only if $\check\omega_{K_t}(s)$ is convex in $(s,t)$. It turns out that the latter need not be true.

\begin{example}\label{n1ex} { Let $n=1$, $D=\D$, $K_0=\{z:\:|z|\le e^{-1}\} $ and $K_1=\{z:\:|z|\le e^{-2}\} $. Then $K_t=\{z:\:|z|\le e^{-1-t}\}$ and the function
$$\check\omega_{K_t}(s)=\max\left\{\frac{s}{1+t},-1\right\}$$
is not convex in $(s,t)$, so $\omega_{K_t}^{}$ is not geodesic.}
It is easy to check that
$$ \check u(s,t)=\max\left\{s, \frac{s+t-1}2,-1\right\},$$
so $K_t=L_t$ and $u_t$ is not a relative extremal function at all.

Note also that
$\BE(\omega_{K_t})=-\Capa(K_t)=-(1+t)^{-1}$ is far from being linear. Finally, $\BE(u_t)=t/2-1$, as expected.
\end{example}

In this example, the geodesics $u_t$ still pertain some features of relative extremal functions.  Namely, recall that a {\it pluriregular condenser} $(K_1,...,K_m, \sigma_1,...,\sigma_m)$ is a system of pluriregular
compact sets $K_m \subset K_{m-1} \subset\ldots\subset K_1 \subset D \subset \overline D =K_0$
and numbers $\sigma_m<\sigma_{m-1}<\ldots<\sigma_1<\sigma_0=0$ such that there is a continuous
plurisubharmonic function $\omega$ on $D$ with zero boundary values,
$K_i = \{z\in D:\: \omega\le \sigma_i\}$ and $\omega$ is maximal on the complement of $K_i$ in the interior of $K_{i-1}$, see \cite{Po}.
In our case, $u_t$ is the extremal function for the condenser
$(K_{1,t},K_{2,t},t-1, -1) $, where $K_{1,t}= \{z:\:|z|\le e^{1-t}\}$ and $K_{2,t}= \{z:\:|z|\le e^{-1-t}\}$, and $\BE(u_t)$ is the energy of the condenser.

It would be nice to know if anything similar holds in the general case of geodesics of relative extremal functions.

\section{Geodesics on $\F_1$}

One cannot apply the above construction to functions from $\F_1(D)$ directly, because they need not be bounded from below and thus existence of the 'good' envelope $\widehat v$ is not guaranteed (in the next section, we will show that generally there are no geodesics for plurisubharmonic functions with nonzero Lelong numbers).

Let $u_j\in\F_1(D)$, $j=0,1$ and let $u_{j,N}\in\E_0(D)$ decrease to $u_j$ as $N\to\infty$. Then their geodesics $u_{t,N}\in\E_0(D)$ linearize the functional $\BE$:
$$\BE(u_{t,N})=(1-t)\,\BE(u_{0,N})+t\,\BE(u_{1,N}).$$
Since $u_{t,N}\ge u_{1}+u_{2}\in\F_1(D)$ for any $N$, the functions $u_{t,N}$ decrease to functions $v_t\in\F_1(D)$  and $\BE(u_{t,N})$ decrease to $\BE(v_{t})$ for $0<t<1$ while $\BE(u_{j,N})$ decrease to $\BE(u_{j})$ for $j=0,1$ by \cite[Thm. 3.8]{Ce03}.
Therefore,
\beq\label{eqvt} \BE(v_{t})=(1-t)\,\BE(u_{0})+t\,\BE(u_{1}).\eeq
Nor also that since $\widehat u_N(z,\zeta)=u_{\log|\zeta|,N}(z)$ satisfy $(dd^c\widehat u_n)^{n+1}=0$ on $D\times S$ and decrease to $\widehat v(z,\zeta)$, we have $(dd^c\widehat v)^{n+1}=0$ as well.

To have a complete analogy with the bounded case, we need to establish the relations $v_t\to u_j$ as $t\to j$ for $j=0$ and $1$. Since $v_t$ are convex in $t$ and $v_t\ge u_1+u_2$, the functions $v_j=\limsup_{t\to j}v_t$ are weak limits of $v_t$ and belong to $\F_1(D)$. By construction, $v_j\le u_j$.

Denote $V_t= (1-t)\,v_0+t\,v_1$. Then a direct computation shows
$\BE(V_t)\to \BE(v_j)$ as $t\to j$. Since $v_t\le V_t$, we get
$$\BE(u_j)=\lim_{t\to j} \BE(v_t)\le \lim_{t\to j} \BE(V_t)=\BE(v_j),$$
which implies $u_j=v_j$ by Corollary~\ref{dompr}.

So, from now on we rename the functions $v_t$ to $u_t$ since they have $u_j$ as there endpoints, for the moment as upper limits when $t\to j$.
We claim that actually $u_t\to u_j$ in capacity, that is, for any $\epsilon>0$, we have $\Capa A_{\epsilon,t}\to 0$, where
$A_{\epsilon,t}=A_{\epsilon,t,j}=\{z\in D:\: |u_t(z)-u_j(z)|>\epsilon\}$. 

We will prove the claim for $j=0$, the case $j=1$ being completely similar.
By subadditivity of the capacity, it suffices to show that $\Capa A_{\epsilon,t}^{\pm}\to 0$, where $A_{\epsilon,t}^+=\{z:\: u_t(z)-u_0(z)>\epsilon\}$ and $A_{\epsilon,t}^-=\{z:\: u_t(z)-u_0(z)<-\epsilon\}$. Moreover, since $u_t\ge u_0+u_1$ and $\Capa\{z:\: u_0(z)+u_1(z)<-N\}=o(N^{-n})$ as $N\to\infty$ \cite[Lemma 2.1]{Ben}, we can assume $u_0+u_1\ge -N$ on $A_{\epsilon,t}$.

Since $u_t\le V_t$, we have $u_1-u_0\ge \epsilon/t$ on $A_{\epsilon,t}^+$, so for any $\psi\in\PSH(D)$,$-1\le\psi\le 0$,
$$\int_{A_{\epsilon,t}^+} (dd^c\psi)^n\le t\,\epsilon^{-1}\int_{A_{\epsilon,t}^+} (u_1-u_0) (dd^c\psi)^n\le
tN\epsilon^{-1}\int_{A_{\epsilon,1/2}^+}  (dd^c\psi)^n\le tN\epsilon^{-1}\Capa A_{\epsilon,1/2}^+$$
for any $t<1/2$, which implies $\Capa A_{\epsilon,t}^+\to 0$ as $t\to 0$.

To work with the set $A_{\epsilon,t}^-$ is more tricky because, in the unbounded case, there are no straightforward subgeodesics with good behavior at the endpoints. We will use here an envelope technique introduced (in the K\"{a}hler setting) in \cite{RWN} and developed in \cite{Da14a} (especially in Theorems~4.3 and 5.2 of the latter paper).

Given $u,v\in \PSH(D)$, denote the largest plurisubharmonic minorant of $\min\{u,v\}$ in $D$ by $P(u,v)$. If $u_0,u_1\in\F_1(D)$ and $C\ge 0$, then $$u_0+u_1\le p_C:= P(u_0,u_1+C)\le u_0,$$ which implies $p_C\in\F_1(D)$. Therefore,
$ w_{t,C}=p_C-Ct\le u_t$ and
$$A_{\epsilon,t}^-\subset \{z:\: w_{t,C}(z)-u_0(z)<-\epsilon\}.$$
Therefore,
$$\lim_{t\to 0}\Capa A_{\epsilon,t}^-\le \inf_{C\ge 0}\Capa B_{\epsilon,C}^-,$$
where $B_{\epsilon,C}^-=\{z:\: p_C(z)-u_0(z)<-\epsilon\}$.

The family $p_C $ increases in $C$ to a function whose upper semicontinuous regularization is a plurisubharmonic function $U$. Moreover, $p_C$ converges to $U$ in capacity, so our claim follows from the lemma below.

\begin{lemma} For any $u_0,u_1\in\F_1(D)$, let $p_C:= P(u_0,u_1+C)$ and let $U$ be the upper semicontinuous regularization of $\lim_{C\to\infty}p_C$. Then $U=u_0$ on $D$.
\end{lemma}

\begin{proof} Our arguments are close to the proof of \cite[Thm. 4.3]{Da14a}. 
Precisely as for $\omega$-plurisubharmonic functions in \cite[Prop.~3.3]{Da14a}, we have for any bounded plurisubharmonic functions $u$ and $v$ the inequality
$$(dd^cP(u,v))^n\le {\bf 1}_{\{P=u\}} (dd^cu)^n+ {\bf 1}_{\{P=v\}}(dd^cv)^n.$$
Let $u_{0,N},\, u_{1,N}\in\E_0(D)$  be approximations of $u_0,u_1\in\F_1(D)$ from the definition of the class $\F_1$, and set $u=u_{0,N}$, $v=u_{1,N}+C$ for $C>0$. Then
$$ {\bf 1}_{\{P=u\}} (dd^cu)^n\le (dd^cu_{0,N})^n$$
and
$$ {\bf 1}_{\{P=v\}}(dd^cv)^n\le {\bf 1}_{\{u_{1,N}<-C\}}(dd^c u_{1,N})^n.$$
For any positive test function $\eta\in C_0(D)$, we have
\begin{eqnarray*} \int_D \eta\, {\bf 1}_{\{u_{1,N}<-C\}}(dd^c u_{1,N})^n &\le &
\frac1C\int_D \eta\, |u_{1,N}|{\bf 1}_{\{u_{1,N}<-C\}}(dd^c u_{1,N})^n \\
&\le &
\frac{\max\eta}{C}|\BE(u_{1,N})| \rightarrow \frac{\max\eta}{C}|\BE(u_{1})|
\end{eqnarray*}
as $N\to\infty$.
Since  $P(u,v)$ decreases to $p_C $ as $N\to \infty$, we get then
$$\int_D \eta\,(dd^cp_C)^n\le  \int_D \eta\,(dd^cu_{0})^n+ \frac{\max\eta}{C}|\BE(u_{1})|,$$
and with $C\to\infty$ we deduce
$(dd^cU)^n\le  (dd^cu_{0})^n$.
Both $U$ and $ u_0$ belong to $\F_1(D)$, so this implies, by \cite[Thm. 4.5]{Ce03}, $U\ge u_0$. Since $U\le u_0$, the proof of the lemma is complete.
\end{proof}

\medskip

We summarize the results of this section as follows.

\begin{theorem}\label{mabE1} For any pair $u_0,u_1\in\F_1(D)$ there exists a geodesic $u_t\subset\F_1(D)$, $0< t< 1$, such that
$u_t$ converge in capacity to $ u_j$ as $t$ approaches $ j=0$ and $j=1$.
The energy functional $v\mapsto\BE(v)$ is concave on $\F_1(D)$, while the function $t\mapsto \BE(u_t)$ is linear on geodesics $u_t$ and convex on subgeodesics $v_t\in\F_1(D)$.
\end{theorem}

\begin{corollary}\label{uniq1} The uniqueness result of Theorem~\ref{uniq} remains true for $u_0,u_1\in\F_1(D)$. \end{corollary}

\section{Case of strong singularities}
The Monge-Amp\`ere current $(dd^cu)^n$ of functions from the class $\F_1$ cannot charge pluripolar sets.
If functions $u_j\in\PSH^-(D)$ are allowed to have stronger singularities, the process of constructing geodesics generally fails. The breaking point is that the presumed 'geodesic' $u_t$ can have $\lim_{t\to j}u_t<u_j$.

\medskip

We start with a simple observation. Let $a\in D$ and let $G_a$ be the pluricomplex Green function of $D$ with pole at $a$.

\begin{lemma}\label{lem:dim1} If $\Phi\in \PSH^-(D\times S)$ is such that
$\limsup\Phi(z,\zeta)\le G_a(z)$ for all $ z\in D$ as $|\zeta|\to e$,
then $\Phi(z,\zeta)\le G_a(z)$ for all $z\in D$ and all $\zeta\in S$.
\end{lemma}

\begin{proof}
The functions $\psi_N(z,\zeta)= \max\{G_a(z),-N\log|\zeta|\}\in\PSH^-(D\times S)$ are equal to $0$ on $\partial D\times S$.  We also have $\psi_N(z,\zeta)\to u_{N,0}(z)=0$ when $|\zeta|\to 1$, and
$\psi_N(z,\zeta)\to u_{N,1}(z)=\max\{G_a(z),-N\}$ when $|\zeta|\to e$.

Furthermore,  they satisfy $(dd^c \psi_N)^{n+1}=0$ everywhere in $D\times S$. Therefore, $\psi_{N,t}$ is the geodesic for $u_{N,0}$ and $u_{N,1}$. Since $\Phi\le\psi_N$ for any $N$, the proof is complete.
\end{proof}

\medskip

A bit more generally, let $u\in\PSH^-(D)$ be such that $A=\{z:\: u(z)=-\infty\}$ is a closed subset of $D$ and $u\in L_{loc}^\infty(D\setminus A)$.
Then the function
$$g_{u}(z)=\limsup_{x\to z}\,\sup\{v(x):\: v\in\PSH^-(D),\ v\le u+O(1)\}$$
is plurisubharmonic in $D$, locally bounded outside $A$ and satisfying $(dd^cg_u)^n=0$ there. When $A$ is a single point, then $g_u\equiv 0$ if and only if $(dd^c u)^n(A)=0$ \cite{R7}.

As is easy to see, $g_u\not\equiv 0$ if $u$ has nonzero Lelong number at some point of $A$; we do not know if the converse is true.

\medskip

By repeating the arguments of the proof of Lemma~\ref{lem:dim1},
we get

\begin{theorem}\label{minenv}If $\Phi\in \PSH^-(D\times S)$ is such that
\beq\label{bcsing} \limsup_{\log|\zeta|\to j}\Phi(z,\zeta)\le u_j(z)\quad\forall z\in D,\ j=0,1,\eeq
then $\Phi(z,\zeta)\le P(z)$ for all $\zeta\in S$, where $P=P(g_{u_0},g_{u_1})$ is the largest plurisubharmonic minorant of the function $\min_j g_{u_j}$.
In particular, if each $u_j=g_{u_j}$, then the largest $\Phi$ satisfying (\ref{bcsing}) coincides with $P$ (and thus is independent of $\zeta$.)
\end{theorem}

\begin{example} Let $A$ be a finite subset of $D$ and let $u_j$ equal the multi-pole Green function of $A$ with weights $m_{j,k}\ge 0$ at $a_k\in A$. Then the best function $\Phi$ satisfying (\ref{bcsing}) is the multi-pole Green function of $A$ with weights $M_k=\max_j m_{j,k}$ at $a_k\in A$.
\end{example}

\medskip\noindent
{\bf Remark.}
 The situation changes if one replaces the segment $0< t< 1$ with the ray $-\infty<t< 0$. For example, let $\vph_j=u_j+w_j$ such that $u_j\in\E_1(D)$ and $w_0=w_1+w$, where $w\in\PSH^-(D)$ has zero boundary values. If $u_t$, $0<t<1$, is the geodesic arc for $u_0$ and $u_1$, then
$$\vph_t=u_{e^t} + w_1+\max\{w, t\}, \quad -\infty<t<0,$$ is a subgeodesic ray with $\vph_t\to\vph_j$ as $t\to \log j$, $j=0, 1$.

\section{Relations to the K\"{a}hler case}

Let $(X,\omega)$ be a compact K\"{a}hler manifold. An upper semicontinuous function $\vph$ on $X$ is called {\it $\omega$-plurisubharmonic} if $\omega+dd^c\vph\ge 0$. Cegrell's classes were generalized to such functions in \cite{GZ07}. A corresponding class $\E_1(X,\omega)$ was introduced, and it has turned to be a natural frame for studying the Mabuchi functional \cite{BBGZ13}; see also a nice presentation in \cite{G15}, where, in addition, toric geodesics on toric manifolds are considered.

Some of problems studied in recent papers by T. Darvas with co-authors (e.g., \cite{BDL15}, \cite{Da14a}, \cite{Da14b}, \cite{DaR}) in the K\"{a}hler setting are close to those treated here. In particular, Proposition~4.2 from \cite{BDL15} is a complete analog of our Corollary~\ref{dompr}. Theorem~5.2 from \cite{Da14a} characterizes $\omega$-plurisubharmonic functions that can be joined by a weak geodesic in terms of a technique from \cite{RWN}, which is closely related to our Theorem~\ref{minenv}. Finally, we have borrowed the idea of using the envelope technique for proving convergence in capacity in Theorem~\ref{mabE1} from Theorems~4.3 and 4.3 of \cite{Da14a}.

\bigskip
{\small {\bf Acknowledgement.} The author is grateful to Tam\'{a}s Darvas for pointing out recent related results in the K\"{a}hler setting and especially the paper \cite{Da14a}.}

\vskip1cm

Tek/Nat, University of Stavanger, 4036 Stavanger, Norway

\vskip0.1cm

{\sc E-mail}: alexander.rashkovskii@uis.no

\end{document}